\documentclass[a4paper,12pt]{amsart}
\usepackage[margin=3 cm]{geometry}
\usepackage[OT2,OT1]{fontenc}
\newcommand\cyr

\usepackage{fullpage}
\newtheorem{theorem}{Theorem}[section]
\newtheorem{rem} [theorem] {Remark}

\newtheorem{lemma}[theorem]{Lemma}

\newtheorem{ex}[theorem]{Example}
\newcommand{\ovprt}{\overline{\partial}}
\newcommand{\ovli}{\overline}

\numberwithin{equation}{section}

\title{ The $\partial$-complex on the Fock space}

\author{ Friedrich Haslinger}

\thanks{Partially supported by the Austrian Science Fund (FWF) project  P28154.}

 \address{ F. Haslinger: Fakult\"at  f\"ur Mathematik, Universit\"at Wien,
Oskar-Morgenstern-Platz 1, A-1090 Wien, Austria}
\email{ friedrich.haslinger@univie.ac.at}


\begin{document}

\maketitle
\begin{abstract}We study certain densely defined unbounded operators on the Fock space. These are the annihilation and creation operators of quantum mechanics. In several complex variables we have the $\partial$-operator and its adjoint $\partial^*$ acting on $(p,0)$-forms with coefficients in the Fock space. We consider the corresponding $\partial$-complex and study spectral properties of the corresponding complex Laplacian $\tilde \Box = \partial \partial^* + \partial^*\partial.$ Finally we study a more general complex Laplacian $\tilde \Box_D = D D^* + D^* D,$ where $D$ is a differential operator of polynomial type, to find the canonical solutions to the inhomogeneous equations $Du=\alpha$ and $D^*v=\beta.$

\end{abstract}

\vskip 1.5 cm

\section{Introduction}
\vskip 1 cm
Purpose of this paper is to consider the $\partial$-complex and to use the powerful classical methods of the $\overline \partial$-complex based on the theory of unbounded densely defined operators on Hilbert spaces, see \cite{Has10}, \cite{Str}.
The main difference to the classical theory is that the underlying Hilbert space is now not an $L^2$-space but a closed subspace of an $L^2$-space - the Fock space of entire functions $A^2(\mathbb C^n, e^{-|z|^2}).$
It is well known that the differentiation with respect to $z_j$ defines an unbounded operator on $A^2(\mathbb C^n, e^{-|z|^2}).$ We consider the operator
$$\partial f = \sum_{j=1}^n \frac{\partial f}{\partial z_j}\, dz_j,$$
which is densely defined on $A^2(\mathbb C^n, e^{-|z|^2})$ and maps to $A^2_{1,0}(\mathbb C^n, e^{-|z|^2})$,
the space of $(1,0)$-forms with coefficients in $A^2(\mathbb C^n, e^{-|z|^2}).$ In general, we get the $\partial$-complex
$$  A^2_{(p-1,0)}(\mathbb{C}^n, e^{-|z|^2}) 
\underset{\underset{\partial^* }
\longleftarrow}{\overset{\partial }
{\longrightarrow}} A^2_{(p,0)}(\mathbb{C}^n, e^{-|z|^2}) \underset{\underset{\partial^* }
\longleftarrow}{\overset{\partial }
{\longrightarrow}} A^2_{(p+1,0)}(\mathbb{C}^n, e^{-|z|^2}),
$$ 
where $1\le p \le n-1$ and $\partial^*$ denotes the adjoint operator of $\partial.$
 
We will choose the domain ${\text{dom}}(\partial)$ in such a way that $\partial$ becomes a closed operator on 
$A^2(\mathbb C^n, e^{-|z|^2}).$  In addition  we get that the corresponding complex Laplacian
$$\tilde \Box_p = \partial^* \partial + \partial \partial^*,$$
with ${\text{dom}} (\tilde\Box_p) =\{ f\in {\text{dom}}(\partial ) \cap {\text{dom}}(\partial^* ) : \partial f \in {\text{dom}}(\partial ^*) \ {\text{and}} \ \partial f^* \in {\text{dom}}(\partial )\}$ 
 acts as unbounded self-adjoint operator on $A^2_{(p,0)}(\mathbb{C}^n, e^{-|z|^2}).$ 
 We point out that in this case the complex Laplacian is  a differential operator of order one. Nevertheless we can use the general features of a Laplacian for these differential operators of order one.
 
 Using an estimate which is analogous to the basic estimate for the $\overline \partial$-complex, we obtain that $\tilde \Box_p$ has a bounded invers $\tilde N_p: A^2_{(p,0)}(\mathbb{C}^n, e^{-|z|^2}) \longrightarrow A^2_{(p,0)}(\mathbb{C}^n, e^{-|z|^2})$ and we show that $\tilde N_p$ is even compact. In addition we compute the spectrum of $\tilde \Box_p.$
 
 \vskip 0.3 cm
 The inspiration for this comes from quantum mechanics, where the annihilation operator $a_j$ can be represented by the differentiation with respect to $z_j$ on $A^2(\mathbb C^n, e^{-|z|^2})$ and its adjoint, the creation operator $a^*_j,$ by the multiplication by $z_j,$ both operators being unbounded densely defined, see \cite{F1}, \cite{FY}. One can show that $A^2(\mathbb C^n, e^{-|z|^2})$ with this action of the $a_j$ and $a^*_j$ is an irreducible representation $M$ of the Heisenberg group, by the Stone-von Neumann theorem it is the only one up to unitary equivalence. Physically $M$ can be thought of as the Hilbert space of a harmonic oscillator with $n$ degrees of freedom and Hamiltonian operator
$$H= \sum_{j=1}^n \frac{1}{2} (P_j^2+Q_j^2) =  \sum_{j=1}^n \frac{1}{2} (a_j^* a_j+a_ja_j^*).$$

 In the second part we consider a general plurisubharmonic weight function $\varphi :\mathbb C^n \longrightarrow \mathbb R$ and the corresponding weighted space of entire functions $A^2(\mathbb C^n, e^{-\varphi}).$ The $\partial$-complex has now the form 
 $$ A^2_{(p-1,0)}(\mathbb{C}^n, e^{-\varphi}) 
\underset{\underset{\partial^*_\varphi }
\longleftarrow}{\overset{\partial }
{\longrightarrow}} A^2_{(p,0)}(\mathbb{C}^n, e^{-\varphi}) \underset{\underset{\partial^*_\varphi }
\longleftarrow}{\overset{\partial }
{\longrightarrow}} A^2_{(p+1,0)}(\mathbb{C}^n, e^{-\varphi}),$$
where $1\le p \le n-1.$ 
Finally we prove a formula which is analogous to the Kohn-Morrey formula for the classical $\overline \partial$-complex, see \cite{Has10}, \cite{Str} or \cite{ChSh}. We will show that 
$$\|\partial u\|^2_\varphi + \|\partial^*_\varphi u\|^2_\varphi = \sum_{j,k=1}^n \int_{\mathbb C^n} \left | \frac{\partial u_j}{\partial z_k} \right |^2 \,  e^{-\varphi}\,d\lambda + \sum_{j,k=1}^n \int_{\mathbb C^n} \frac{\partial^2 \varphi}{\partial z_k \partial \overline z_j}\, u_j \overline u_k \,  e^{-\varphi}\,d\lambda + T,$$ 
for $u\in {\text{dom}}(\partial ) \cap {\text{dom}}(\partial^*_\varphi ),$ where the term $T$ is non-positive.
\vskip 0.3 cm
Finally we investigate operators of the $Du = \sum_{j=1}^n p_j (u)\, dz_j,$
where $u\in A^2(\mathbb C^n, e^{-|z|^2})$ and $p_j(\frac{\partial}{\partial z_1}, \dots , \frac{\partial}{\partial z_n})$
are polynomial differential operators with constant coefficients. Differential operators of polynomial type on the Fock space were also investigated by J.D. Newman and H. Shapiro in \cite{NS1} and \cite{NS2} and by H. Render \cite{Ren} in the real analytic setting.
Replacing $\partial$ by $D$ one gets a corresponding complex Laplacian $\tilde \Box_D = DD^*+D^*D,$ for which one can use duality and the machinery of the $\ovprt$-Neumann operator  (\cite{Ko1}, \cite{Ko2}) in order to prove existence and boundedness of the inverse to $\tilde \Box_D$ and to find the canonical solutions to the inhomogeneous equations $Du=\alpha$ and $D^*v=\beta.$

\vskip 1 cm

\section{The Fock space}
\vskip 1cm
We consider the Fock space\index{Fock space} $A^2(\mathbb{C}^n, e^{-|z|^2})$ consisting of all entire functions $f$ such that
$$\|f\|^2= \int_{\mathbb{C}^n} |f(z)|^2 \, e^{-|z|^2}\, d\lambda (z) < \infty.$$
It is clear, that the Fock space is a Hilbert space with the inner product
$$(f,g) =\int_{\mathbb{C}^n} f(z) \, \overline{g(z)} \,  e^{-|z|^2}\, d\lambda (z).$$
Setting $n=1,$ we obtain for $f\in A^2(\mathbb{C}, e^{-|z|^2})$
that
\begin{eqnarray*}
|f(z)|&\le & \frac{1}{\pi r^2}\ \int_{D_r(z)} e^{|w|^2/2} \, \ |f(w)|\, e^{-|w|^2/2}\, d\lambda (w)\\
&\le &
\frac{1}{\pi r^2}\ \left ( \int_{D_r(z)} e^{|w|^2}\, d\lambda (w) \right )^{1/2} \ 
\left ( \int_{D_r(z)} |f(w)|^2\, e^{-|w|^2}\, d\lambda (w) \right )^{1/2}\\
&\le & C \left ( \int_{\mathbb{C}} |f(w)|^2\, e^{-|w|^2}\,d\lambda (w) \right )^{1/2}\\
&\le & C  \|f\| , 
\end{eqnarray*}
where $C$ is a constant only depending on $z.$ 
In addition, for each compact subset $L$ of $\mathbb C$ there exists a constant $C_L>0$ such that 
\begin{equation}\label{comp1}
\sup_{z\in L}|f(z)| \le C_L \, \|f\|,
\end{equation}
for all $f\in  A^2(\mathbb{C}, e^{-|z|^2}).$

For several variables one immediately gets an analogous estimate.
This implies that  the Fock space $A^2(\mathbb{C}^n, e^{-|z|^2})$ has the reproducing property. The monomials $\{z^\alpha \}$ constitute an orthogonal basis, where $\alpha=(\alpha_1, \dots, \alpha_n)\in \mathbb N_0^n$ is a multiindex,  and the norms of the monomials are 
\begin{eqnarray*}
\| z^\alpha \|^2 &=& \int_{\mathbb{C}} |z_1|^{2\alpha_1}\, e^{-|z_1|^2}\, d\lambda(z_1) \dots 
 \int_{\mathbb{C}} |z_n|^{2\alpha_n}\, e^{-|z_n|^2}\, d\lambda(z_n)\\
&=&  (2\pi)^n \, \int_0^\infty r^{2\alpha_1 +1} e^{-r^2}\, dr \dots  \int_0^\infty r^{2\alpha_n +1} e^{-r^2}\, dr\\
&=& \pi^n \alpha_1! \dots \alpha_n! .
\end{eqnarray*}
 It follows that each function $f\in A^2(\mathbb{C}^n, e^{-|z|^2})$ can be written in the form
$$ f = \sum_\alpha f_\alpha \varphi_\alpha,$$
where 
\begin{equation}\label{vons}
\varphi_\alpha (z) = \frac{z^\alpha}{\sqrt{\pi^n \alpha !}}\ \ {\text{and}} \ \ \sum_\alpha |f_\alpha |^2 < \infty
\end{equation}
and $\alpha! = \alpha_1! \dots \alpha_n!.$

Hence the Bergman kernel of $A^2(\mathbb{C}^n, e^{-|z|^2})$ is of the form
\begin{equation}
\label{sec: fockkernel}
K(z,w)= \sum_\alpha \frac{z^\alpha \overline w^\alpha }{\| z^\alpha \|^2}= 
 \frac{1}{\pi^n} \sum_{k=0}^\infty \, \sum_{|\alpha|=k} \,  \frac{z^\alpha \overline w^\alpha}{ \alpha_1! \dots \alpha_n!} = \frac{1}{\pi^n} \, \exp(z_1\overline w_1 + \dots  +z_n \overline w_n).
 \end{equation}
See \cite{Zh} for an extensive study of the Fock space.

\vskip 0.5 cm

We point out that the space $A^2(\mathbb{C}, e^{-|z|^2})$ serves for a representation of the states in quantum mechanics (see \cite{FY}), where
$$a(f) = \frac{df}{dz} $$
is the annihilation operator and
$$a^*(f) = z f$$
is the creation operator, both of them being densely defined unbounded operators on $A^2(\mathbb{C}, e^{-|z|^2}).$ 
The span of the finite linear combinations of the basis functions $\varphi_\alpha$ is dense in $A^2(\mathbb{C}, e^{-|z|^2}).$ Hence both operators $a$ and $a^*$ are densely defined.

The function 
$$F(z) = \sum_{k=2}^\infty \frac{\varphi_k(z)}{\sqrt{k(k-1)}} \in A^2(\mathbb{C}, e^{-|z|^2}),$$
but 
$$F'(z) = \sum_{k=1}^\infty \frac{\varphi_k(z)}{\sqrt{k}} \notin A^2(\mathbb{C}, e^{-|z|^2}),$$
and 
$$G(z) = \sum_{k=0}^\infty \frac{\varphi_k(z)}{k+1} \in A^2(\mathbb{C}, e^{-|z|^2}), $$
but 
$$zG(z) =  \sum_{k=1}^\infty \frac{\varphi_k(z)}{\sqrt{k}} \notin A^2(\mathbb{C}, e^{-|z|^2}),$$
hence both operators $a$ and $a^*$ are unbounded operators on $A^2(\mathbb{C}, e^{-|z|^2}).$

But taking a primitive of a function $f\in  A^2(\mathbb{C}, e^{-|z|^2})$ yields a bounded operator 

$$T: A^2(\mathbb{C}, e^{-|z|^2}) \longrightarrow  A^2(\mathbb{C}, e^{-|z|^2});$$
let 

$$f(z) = \sum_{k=0}^\infty f_k \frac{z^k}{\sqrt{\pi}\sqrt{k!}}, \ \ \sum_{k=0}^\infty |f_k|^2 < \infty.$$
Then the function 

$$h(z) = \sum_{k=0}^\infty f_k \frac{z^{k+1}}{\sqrt{\pi}\sqrt{k+1}\sqrt{(k+1)!}}$$
defines a primitive of $f$ and we can write
$$ T(f) =h= \sum_{k=1}^\infty \frac{1}{\sqrt{k!}} \, (f, \tilde \varphi_k) \varphi_k,$$
where $\tilde \varphi_k =\varphi_{k-1}$ and the constant term in the Taylor series expansion of the primitive is always $0$.
This implies immediately that $T$ is even a compact operator.

This  is also a special result from the theory of Volterra-type integration operators on the Fock space  of the form
$$T_gf (z) = \int_0^z f g' \, d\zeta,$$
see \cite{Oliv}.

\vskip 0.3 cm

Next we define the domain of the operator $a$  to be
$${\text{dom}} (a) = \{ f\in A^2(\mathbb{C}, e^{-|z|^2}) : \, f' \in A^2(\mathbb{C}, e^{-|z|^2}) \}.$$
Then  ${\text{dom}} (a^*)$ consists of all functions $g\in {\text{dom}} (a)$ such that the densely defined linear functional $L(f)=(a(f),g)$ is continuous on ${\text{dom}} (a).$ This implies that there exists a function $h\in  
A^2(\mathbb{C}, e^{-|z|^2}),$ such that $L(f)=(a(f),g)=(f,h).$

Next we show that
$${\text{dom}} (a^*) = \{ g\in A^2(\mathbb{C}, e^{-|z|^2}) : \, zg \in A^2(\mathbb{C}, e^{-|z|^2}) \}.$$

Let $f\in {\text{dom}} (a)$ and $g\in {\text{dom}} (a^*).$
Then
\begin{eqnarray*}
(a(f), g)&=& \int_{\mathbb C} \frac{df(z)}{dz}\, \overline{g(z)}\,e^{-|z|^2} \, d\lambda (z) \\
&=& - \int_{\mathbb C} f(z) \, \frac{d}{dz} ( \overline{g(z)}\,e^{-|z|^2}) \, d\lambda (z) \\
&=&  \int_{\mathbb C} f(z) \,\overline{zg(z)}\,e^{-|z|^2} \, d\lambda (z)\\
&=& (f, a^*(g)),
\end{eqnarray*}
where we used integration by parts in the first step (see \eqref{partint} for a detailed proof) and that 
$$\frac{d}{dz} ( \overline{g(z)}\,e^{-|z|^2})= \overline{g(z)} (-\overline z \, e^{-|z|^2}).$$

An alternative proof uses Taylor series expansion: let 
$$f(z) = \sum_{k=0}^\infty f_k \varphi_k(z) \ \ {\text{and}} \ \ g(z) = \sum_{k=0}^\infty g_k \varphi_k(z),$$
where $(f_k)_k, (g_k)_k \in l^2.$ Then we have
$$
(a(f), g)= \sum_{k=0}^\infty \sqrt{k+1}\, f_{k+1} \, \overline{g_k} =(f, a^*(g)).$$

\begin{rem}\label{1basic}
(a) We point out that the commutator satisfies $[a,a^*]=I,$ which is of importance in quantum mechanics,
see \cite{FY}.

(b) For $f\in {\text{dom}} (a)\cap {\text{dom}} (a^*)$ we get from the last results that
\begin{equation}\label{2basic}
\|a(f)\|^2 + \|a^*(f)\|^2 = 2\|a(f)\|^2 + \|f\|^2.
\end{equation}
\end{rem}

\begin{lemma}\label{closed} The operators $a$ and $a^*$ are densely defined operators on $ A^2(\mathbb{C}, e^{-|z|^2})$ with closed graph.

\end{lemma}

\begin{proof} By general properties of unbounded operators, it suffices to prove the assertion for $a$ (see \cite{Has10} or \cite{Wei}). Let $(f_j)_j$ be a sequence in ${\text{dom}} (a)$ such that $\lim_{j\to \infty}f_j=f$ and $\lim_{j\to \infty} a(f_j)=g.$ We have to show that $f\in {\text{dom}} (a)$ and $a(f)=g.$ By (\ref{comp1}) it follows that $(f_j)_j$ converges uniformly on each compact subset of $\mathbb C$ to $f$ and the same is true for the derivatives $\lim_{j\to \infty}f'_j=f'.$ This implies that $f'=a(f)=g.$ We supposed that $g\in  A^2(\mathbb{C}, e^{-|z|^2})$ and  we know that taking primitives does not leave $ A^2(\mathbb{C}, e^{-|z|^2}) .$ Therefore we get that
$f\in {\text{dom}} (a),$
which proves the assertion.
\end{proof}

For the rest of this section we consider the Fock space in several variables with the weight $\varphi (z)= |z_1|^2+ \dots + |z_n|^2.$
 We will denote the derivative with respect to $z$ by $\partial$ and
in the following we will consider the $\partial$-complex for the Fock space in several variables
 \begin{equation*}
  A^2_{(p-1,0)}(\mathbb{C}^n, e^{-|z|^2}) 
\underset{\underset{\partial^* }
\longleftarrow}{\overset{\partial }
{\longrightarrow}} A^2_{(p,0)}(\mathbb{C}^n, e^{-|z|^2}) \underset{\underset{\partial^* }
\longleftarrow}{\overset{\partial }
{\longrightarrow}} A^2_{(p+1,0)}(\mathbb{C}^n, e^{-|z|^2}),
\end{equation*}

where $A^2_{(p,0)}(\mathbb{C}^n, e^{-|z|^2})$ denotes the Hilbert space of $(p,0)$-forms with coefficients in 
$A^2(\mathbb{C}^n, e^{-|z|^2}), $ and 
$$\partial f = \sum_{|J|=p}\, ' \, \sum_{j=1}^n \frac{\partial f_J}{\partial  z_j}\, d z_j \wedge d z_J$$ 
for a $(p,0)$-form 
$$f=\sum_{|J|=p}\, ' \, f_J\,  d z_J$$ 
with summation over increasing multiindices $J=(j_1, \dots, j_p), \ 1\le p\le n-1;$
and we take
$${\text{dom}}(\partial ) = \{ f\in A^2_{(p,0)}(\mathbb{C}^n, e^{-|z|^2}) : \partial f \in A^2_{(p+1,0)}(\mathbb{C}^n, e^{-|z|^2}) \}.$$
Now let
$$\tilde\Box_p = \partial^* \partial + \partial \partial^*,$$
with ${\text{dom}} (\tilde\Box_p) =\{ f\in {\text{dom}}(\partial ) \cap {\text{dom}}(\partial^* ) : \partial f \in {\text{dom}}(\partial ^*) \ {\text{and}} \ \partial^* f \in {\text{dom}}(\partial )\}.$ 

Then $\tilde\Box_p $ acts as unbounded self-adjoint operator on $A^2_{(p,0)}(\mathbb{C}^n, e^{-|z|^2}),$ see \cite{Has10}.

\begin{rem}
(a) It  is pointed  out  that  a  $(1,0)$-form
$g=\sum_{j=1}^{n}g_j\,d z_j$ with
holomorphic  coefficients  is  invariant under the pull back
by  a  holomorphic  map  $F = (F_1,\dots, F_n):
\mathbb C^n \longrightarrow \mathbb C^n .$ We have
$$F^*g = \sum_{l=1}^n g_l \, d F_l =  \sum_{j=1}^{n}\left ( \sum_{l=1}^n g_l
\frac{\partial  F_l}{\partial  z_j} \right )
\,d  z_j,$$
where we used the fact that
$$d F_l = \partial  F_l + \ovprt  \,  F_l =  \sum_{j=1}^{n} \frac{\partial  F_l}{\partial z_j}\,d z_j + \sum_{j=1}^{n} \frac{\partial  F_l}{\partial \overline z_j}\, d\overline z_j = \sum_{j=1}^{n} \frac{\partial  F_l}{\partial z_j}\, d z_j .$$
 The expressions
$\frac{\partial F_l}{\partial  z_j}$ are  holomorphic.

\vskip 0.3 cm

(b) For $p=0$ and a function $f\in{\text{dom}} (\tilde\Box_0)$ we have 
$$\tilde\Box_0 f = \partial^* \partial f = \sum_{j=1}^nz_j \frac{\partial f}{\partial z_j}.$$

If $1\le p \le n-1$ and $f=\sum_{|J|=p}\, ' \, f_J\,  d z_J \in {\text{dom}} (\tilde\Box_p)$ is a $(p,0)$-form, we have
\begin{equation}\label{spec3}
\tilde \Box_p f = \sum_{|J|=p}\, ' \,  ( \sum_{k=1}^n z_k \frac{\partial f_J}{\partial z_k}+pf_J)\, dz_J.
\end{equation}

 For $p=n$ and a $(n,0)$-form $F \in {\text{dom}} (\tilde\Box_n)$ (here we identify the $(n,0)$-form with a function), we have 
$$\tilde\Box_n F = \partial \partial^* F = \sum_{j=1}^nz_j \frac{\partial F}{\partial z_j} + nF.$$
\end{rem}

Before we continue the study of the box operator $\tilde\Box_p,$ we collect some facts about Fock spaces with more general weights.
\vskip 1.5 cm

\section{ Generalized Fock spaces}
\vskip 1.5 cm

Let $\varphi : \mathbb C^n \longrightarrow \mathbb R$ be a plurisubharmonic $\mathcal C^\infty$ function.
Let 
$$A^2(\mathbb C^n, e^{-\varphi}) = \{ f: \mathbb C^n \longrightarrow \mathbb C \ {\text{entire}} : 
\|f \|_\varphi^2=\int_{\mathbb C^n} |f|^2 \, e^{-\varphi}\, d\lambda <\infty \},$$
with inner product 
$$(f,g)_\varphi = \int_{\mathbb C^n} f \, \overline g \, e^{-\varphi}\, d\lambda.$$
It is easily seen that $A^2(\mathbb C^n, e^{-\varphi})$ is a Hilbert space with the reproducing property. Hence it has a reproducing kernel $K_\varphi (z,w)$ (Bergman kernel) which has the following properties:
$K_\varphi (w,z) =\overline{K_\varphi (z,w)},$ the function $z\mapsto K_\varphi (z,w)$ belongs to $A^2(\mathbb C^n, e^{-\varphi})$ and
$$f(z) = \int_{\mathbb C^n} K_\varphi (z,w) \, f(w) \, e^{-\varphi (w)}\, d\lambda(w),$$
for each $f\in A^2(\mathbb C^n, e^{-\varphi}).$

The Bergman projection $P_\varphi : L^2(\mathbb C^n, e^{-\varphi}) \longrightarrow A^2(\mathbb C^n, e^{-\varphi})$ can be written in the form
$$P_\varphi F (z) =  \int_{\mathbb C^n} K_\varphi (z,w) \, F(w) \, e^{-\varphi (w)}\, d\lambda(w),$$
for $F\in L^2(\mathbb C^n, e^{-\varphi}).$
\begin{rem}
We indicate that $A^2(\mathbb C^n, e^{-\varphi})$ is infinite dimensional, if the lowest eigenvalue $\mu_\varphi$ of the Levi matrix 
$$\left ( \frac{\partial^2 \varphi}{\partial z_k \partial \overline z_j} \right )_{j,k=1}^n$$ 
satisfies $\lim_{|z|\to \infty} |z|^2\, \mu_\varphi (z) = +\infty,$ see \cite{Shi} or \cite{Has10}.
\end{rem} 

\vskip 0.3 cm

We study the $\partial$-complex
 \begin{equation*}
  A^2_{(p-1,0)}(\mathbb{C}^n, e^{-\varphi}) 
\underset{\underset{\partial^*_\varphi }
\longleftarrow}{\overset{\partial }
{\longrightarrow}} A^2_{(p,0)}(\mathbb{C}^n, e^{-\varphi}) \underset{\underset{\partial^*_\varphi }
\longleftarrow}{\overset{\partial }
{\longrightarrow}} A^2_{(p+1,0)}(\mathbb{C}^n, e^{-\varphi}),
\end{equation*}

where $A^2_{(p,0)}(\mathbb{C}^n, e^{-\varphi})$ denotes the Hilbert space of $(p,0)$-forms with coefficients in 
$A^2(\mathbb{C}^n, e^{-\varphi}), $ and 
$$\partial f = \sum_{|J|=p}\, ' \, \sum_{j=1}^n \frac{\partial f_J}{\partial  z_j}\, d z_j \wedge d z_J$$ 
for a $(p,0)$-form 
$$f=\sum_{|J|=p}\, ' \, f_J\,  d z_J$$ 
with summation over increasing multiindices $J=(j_1, \dots, j_p), \ 1\le p\le n-1;$
and we take
$${\text{dom}}(\partial ) = \{ f\in A^2_{(p,0)}(\mathbb{C}^n, e^{-\varphi} ): \partial f \in A^2_{(p+1,0)}(\mathbb{C}^n, e^{-\varphi}) \}.$$

The adjoint operator to $\partial$ depends on the weight:
$${\text{dom}}(\partial^*_\varphi ) = \{ g \in A^2_{(p+1,0)}(\mathbb{C}^n, e^{-\varphi}) : f \mapsto (\partial f,g)_\varphi \ {\text{is continuous on}} \, {\text{dom}}(\partial )\}.$$

We use the Gau\ss--Green Theorem in order to compute the adjoint $\partial^*_\varphi.$

Let $\Omega =\{ z\in \mathbb{C}^n \ : \ r(z)<0 \},$ where $r$ is a real valued $\mathcal{C}^1$-function with
$$\nabla_z r := ( \frac{\partial r}{\partial z_1}, \dots , \frac{\partial r}{\partial z_n})\neq 0$$
on $b\Omega = \{ z \, : \, r(z) =0 \}.$
Without loss of generality we can suppose that $|\nabla_z r|= |\nabla r |=1 $ on $b \Omega. $ For $u,v \in \mathcal{C}^\infty (\overline \Omega )$ and
$$(u,v)= \int_\Omega u(z) \overline{v(z)} \, d\lambda (z).$$
The Gau\ss--Green Theorem implies that
\begin{equation}\label{eq: partint}\left ( \frac{\partial u}{\partial z_k}, v \right ) = - \left ( u, \frac{\partial v}{\partial \overline z_k} \right )+ \int_{b\Omega }u(z)\, \overline{v(z)}\,\frac{\partial r}{\partial z_k}(z) \,d\sigma (z),
\end{equation}
where $d\sigma $ is the surface measure on $b\Omega .$
 
In our case we have holomorphic components  $f_J$ and $g_{jJ}$ and the inner product
$$\left (\frac{\partial f_J}{\partial z_j}, g_{jJ} \right )_\varphi = \int_{\mathbb C^n}\frac{\partial f_J}{\partial z_j}
\, \overline{ g_{jJ}}\, e^{-\varphi}\, d\lambda.$$
Now let $\Omega = \{ z : |z| <R\}$ and take $r(z)= \frac{|z|^2-R^2}{R}$ and apply  \eqref{eq: partint} to get 
\begin{equation}\label{pint}
\int_{|z|\le R} \frac{\partial f_J}{\partial z_j} \, \overline{ g_{jJ}}\, e^{-\varphi}\, d\lambda
- \int_{|z|\le R} f_J \, \overline{ \frac{\partial \varphi}{\partial z_j} g_{jJ}}\, e^{-\varphi}\, d\lambda
= \int_{|z|=R} f_J \, \overline{g_{jJ}}\, \frac{\overline{z_j}}{R}\, e^{-\varphi}\, d\sigma.
\end{equation}

By Cauchy-Schwarz we get
$$| \int_{|z|=R} f_J \, \overline{g_{jJ}}\, \frac{\overline{z_j}}{R}\, e^{-\varphi}\, d\sigma |^2 \le
 \int_{|z|=R} |f_J|^2\, e^{-\varphi}\, d\sigma \,   \int_{|z|=R} |g_{jJ}|^2\, e^{-\varphi}\, d\sigma,$$
 and as 
$$\|f_J\|^2_\varphi = \int_{\mathbb C^n} |f_J|^2\, e^{-\varphi}\, d\lambda = \int_0^\infty R^{2n-1} \int_{|z|=R} |f_J|^2 \, e^{-\varphi}\, d\sigma \, dR < \infty$$
the right hand side of \eqref{pint} tends to zero as $R$ tends to $\infty.$
So we obtain for the components of the $(p,0)$-form $f$ and the $(p+1,0)$-form $g$
\begin{eqnarray*}
\left (\frac{\partial f_J}{\partial z_j}, g_{jJ} \right )_\varphi &=& \left (f_J, \frac{\partial \varphi}{\partial \overline z_j}\, g_{jJ} \right )_\varphi \\
&=& \left (P_\varphi(f_J), \frac{\partial \varphi}{\partial \overline z_j}\, g_{jJ} \right )_\varphi \\
&=& \left (f_J,P_\varphi( \frac{\partial \varphi}{\partial \overline z_j}\, g_{jJ}) \right )_\varphi ,
\end{eqnarray*}
where we used the fact that the components $f_J$ are holomorphic. Hence we obtain
\begin{equation}\label{partint}
\partial^*_\varphi u = \sum_{|K|=p-1}\kern-7pt{}^{\prime}  \, \, \sum_{j=1}^n P_\varphi (\frac{\partial \varphi}{\partial \overline z_j}\, u_{jK})\,dz_K ,
\end{equation}
for a $(p,0)$-form $u\in {\text{dom}}(\partial^*_\varphi ).$

Similar to Lemma \ref{closed} one shows that $\partial : {\text{dom}}(\partial) \longrightarrow  A^2_{(p+1,0)}(\mathbb{C}^n, e^{-\varphi}) $ has closed graph.

Now let
$$\tilde\Box = \partial^*_\varphi \partial + \partial \partial^*_\varphi,$$
with ${\text{dom}} (\tilde\Box) =\{ f\in {\text{dom}}(\partial ) \cap {\text{dom}}(\partial^*_\varphi ) : \partial f \in {\text{dom}}(\partial ^*_\varphi) \ {\text{and}} \ \partial^*_\varphi f \in {\text{dom}}(\partial )\}.$ 

Then $\tilde\Box $ acts as unbounded self-adjoint operator on $A^2_{(p,0)}(\mathbb{C}^n, e^{-\varphi}),$ see \cite{ChSh}, \cite{Has10}, \cite{Str}.

\vskip 0.4 cm

In the following we prove an identity which is analogous to the Kohn-Morrey formula for the $\ovprt$-complex (see \cite{Str}, \cite{Has10}). Now an additional non-positive term appears, which vanishes for the weighted $\ovprt$-complex.
 
\begin{theorem}\label{lbasic1}
Let $u = \sum_{j=1}^n u_j \, dz_j \in A^2_{(1,0)}(\mathbb{C}^n, e^{-\varphi})$ and suppose that
 $u\in {\text{dom}}(\partial ) \cap {\text{dom}}(\partial^*_\varphi ).$ Then
 \begin{eqnarray*}
\|\partial u\|^2_\varphi + \|\partial^*_\varphi u\|^2_\varphi &=& \sum_{j,k=1}^n \int_{\mathbb C^n} \left | \frac{\partial u_j}{\partial z_k} \right |^2 \,  e^{-\varphi}\,d\lambda + \sum_{j,k=1}^n \int_{\mathbb C^n} \frac{\partial^2 \varphi}{\partial z_k \partial \overline z_j}\, u_j \overline u_k \,  e^{-\varphi}\,d\lambda \\
&-&  \sum_{j,k=1}^n \left (  \frac{\partial \varphi}{\partial \overline z_j}\,u_j-P_\varphi(\frac{\partial \varphi}{\partial \overline z_j}\,u_j) , \frac{\partial \varphi}{\partial \overline z_k}\,u_k \right )_\varphi.
 \end{eqnarray*}
\end{theorem}

\begin{proof}
We get since 
$$\partial u = \sum_{j<k} \left ( \frac{\partial u_j}{\partial z_k}-
\frac{\partial u_k}{\partial z_j}\right )\,dz_j \wedge dz_k \ {\text{and}} \ 
\partial^*_\varphi u= \sum_{j=1}^n P_\varphi (\frac{\partial \varphi}{\partial \overline z_j}\, u_j )    $$
that
$$
\| \ovprt u \|^2_\varphi + \| \partial ^*_\varphi u\|^2_\varphi = 
\int_{\mathbb{C}^n}\sum_{j<k} \left |\frac{\partial u_j}{\partial z_k}-
\frac{\partial u_k}{\partial z_j}\right |^2 \, e^{-\varphi}\,d\lambda +
\int_{\mathbb{C}^n} \sum_{j,k =1}^n P_\varphi (\frac{\partial \varphi}{\partial \overline z_j}u_j)\,
\overline{P_\varphi(\frac{\partial \varphi}{\partial \overline z_k} u_k)}\,  e^{-\varphi}\,d\lambda $$
$$=\sum_{j,k =1}^n \int_{\mathbb{C}^n} \left |\frac{\partial u_j}{\partial z_k}
\right |^2 \,e^{-\varphi}\,d\lambda +
\sum_{j,k=1}^n  \int_{\mathbb{C}^n} \left ( P_\varphi(\frac{\partial \varphi}{\partial \overline z_j}\,u_j)
\overline{P_\varphi(\frac{\partial \varphi}{\partial \overline z_k} u_k)} - \frac{\partial u_j}{\partial z_k}\, 
\overline{\frac{\partial u_k}{\partial z_j}}\, \right ) \,e^{-\varphi}\,d\lambda $$
$$=\sum_{j,k =1}^n \int_{\mathbb{C}^n}  \left |\frac{\partial u_j}{\partial z_k}
\right |^2 \,e^{-\varphi}\,d\lambda + \sum_{j,k =1}^n  \int_{\mathbb{C}^n} \left [ \frac{\partial}{\partial z_k},P_\varphi \circ \frac{\partial \varphi}{\partial \overline z_j}\right ] \,u_j \, \overline{u}_k\, 
e^{-\varphi}\,d\lambda,$$

where we used the fact that for $f,g\in A^2(\mathbb{C}^n, e^{-\varphi})$ we have
$$\left ( \frac{\partial f}{\partial z_k}, g\right )_\varphi =
 \left ( f,P_\varphi(\frac{\partial \varphi}{\partial \overline z_k}g) \right )_\varphi
$$
and hence
$$
\left ( \left [ \frac{\partial}{\partial z_k},P_\varphi \circ \frac{\partial \varphi}{\partial \overline z_j}\right ] \,u_j , u_k \right )_\varphi =\left ( P_\varphi(\frac{\partial \varphi}{\partial \overline z_j}\,u_j), P_\varphi( \frac{\partial \varphi}{\partial \overline z_k}\,u_k) \right )_\varphi- \left (\frac{\partial u_j}{\partial z_k}, \frac{\partial u_k}{\partial  z_j} \right )_\varphi .$$
Since  we have
\begin{equation}\label{tors2}
\left ( \left [ \frac{\partial}{\partial z_k},P_\varphi \circ \frac{\partial \varphi}{\partial \overline z_j}\right ] \,u_j ,u_k \right )_\varphi =
\left ( \left [ \frac{\partial}{\partial z_k}, P_\varphi \right ] (\frac{\partial \varphi}{\partial \overline z_j}\,u_j) , u_k \right )_\varphi + \left ( \frac{\partial^2 \varphi}{\partial z_k \partial \overline z_j}\, u_j, u_k \right )_\varphi
 \end{equation}
and we have 
\begin{eqnarray*}
\left ( \left [ \frac{\partial}{\partial z_k}, P_\varphi \right ] (\frac{\partial \varphi}{\partial \overline z_j}\,u_j) , u_k \right )_\varphi 
&=& \left ( P_\varphi(\frac{\partial \varphi}{\partial \overline z_j}\,u_j),  P_\varphi(\frac{\partial \varphi}{\partial \overline z_k}\,u_k) \right )_\varphi - \left (\frac{\partial \varphi}{\partial \overline z_j}\,u_j,\frac{\partial \varphi}{\partial \overline z_k}\,u_k \right )_\varphi \\
&=&\left ( P_\varphi(\frac{\partial \varphi}{\partial \overline z_j}\,u_j) - \frac{\partial \varphi}{\partial \overline z_j}\,u_j,\frac{\partial \varphi}{\partial \overline z_k}\,u_k \right )_\varphi ,
\end{eqnarray*}
so we get the desired result.
\end{proof}

\begin{rem} 
The last term
$$\sum_{j,k=1}^n \left (  \frac{\partial \varphi}{\partial \overline z_j}\,u_j-P_\varphi(\frac{\partial \varphi}{\partial \overline z_j}\,u_j) , \frac{\partial \varphi}{\partial \overline z_k}\,u_k \right )_\varphi$$
vanishes for $\varphi (z) = |z_1|^2 +\dots + |z_n|^2,$ and
we obtain
\begin{equation}\label{eq: komo5}
\| \partial u \|^2_\varphi + \| \partial^*_\varphi u\|^2_\varphi =
\sum_{j,k =1}^n \int_{\mathbb{C}^n}  \left |\frac{\partial u_j}{\partial z_k}
\right |^2 \,e^{-|z|^2}\,d\lambda + \sum_{j =1}^n  \int_{\mathbb{C}^n}|u_j|^2\, e^{-|z|^2}\,d\lambda
\end{equation}

If $n=1$ and $u$ is a $(1,0)$-form, we have $\partial u=0$ and 
$$\|\partial u\|^2_\varphi + \|\partial^*_\varphi u\|^2_\varphi = \|\partial^*_\varphi u\|^2 _\varphi= \|u\|^2_\varphi+ \|u'\|_\varphi ^2.$$ 
\end{rem}

\vskip 0.3 cm
\begin{theorem}\label{torsion}
The last term in Theorem \ref{lbasic1} 
$$  \sum_{j,k=1}^n \left (  \frac{\partial \varphi}{\partial \overline z_j}\,u_j-P_\varphi(\frac{\partial \varphi}{\partial \overline z_j}\,u_j) , \frac{\partial \varphi}{\partial \overline z_k}\,u_k \right )_\varphi$$
is always non-negative; we have 
$$  \sum_{j,k=1}^n \left (  \frac{\partial \varphi}{\partial \overline z_j}\,u_j-P_\varphi(\frac{\partial \varphi}{\partial \overline z_j}\,u_j) , \frac{\partial \varphi}{\partial \overline z_k}\,u_k \right )_\varphi $$
  $$=  \sum_{j,k=1}^n \left ( \left [ \frac{\partial}{\partial z_k}, \frac{\partial \varphi}{\partial \overline z_j}\right ] \,u_j , u_k \right )_\varphi - 
 \sum_{j,k=1}^n \left (  \left [ \frac{\partial}{\partial z_k},P_\varphi \circ \frac{\partial \varphi}{\partial \overline z_j}\right ] \,u_j , u_k \right )_\varphi $$
 
 $$= \| R_\varphi v_1+ \dots + R_\varphi v_n \|^2_\varphi 
 = \|V\|^2_\varphi - \|P_\varphi V\|^2_\varphi,$$

where $R_\varphi$ denotes the orthogonal projection $R_\varphi = I -P_\varphi $ and 
$$V=\sum_{j=1}^n v_j = \sum_{j=1}^n\frac{\partial \varphi}{\partial \overline z_j}\,u_j.$$
\end{theorem}

\begin{proof}
Since 
$$\sum_{j,k=1}^n \left ( \left [ \frac{\partial}{\partial z_k}, \frac{\partial \varphi}{\partial \overline z_j}\right ] \,u_j , u_k \right )_\varphi = \sum_{j,k=1}^n \int_{\mathbb C^n} \frac{\partial^2 \varphi}{\partial z_k \partial \overline z_j}\, u_j \overline u_k \,  e^{-\varphi}\,d\lambda,$$
we get from \eqref{tors2} that 
$$ \sum_{j,k=1}^n \left ( \left [ \frac{\partial}{\partial z_k}, \frac{\partial \varphi}{\partial \overline z_j}\right ] \,u_j , u_k \right )_\varphi - 
 \sum_{j,k=1}^n \left (  \left [ \frac{\partial}{\partial z_k},P_\varphi \circ \frac{\partial \varphi}{\partial \overline z_j}\right ] \,u_j , u_k \right )_\varphi $$
 $$=-\sum_{j,k=1}^n \left (  \left [ \frac{\partial}{\partial z_k},P_\varphi \right ] \,(\frac{\partial \varphi}{\partial \overline z_j}\,u_j), u_k \right )_\varphi ,$$
 which equals
$$  \sum_{j,k=1}^n \left (  \frac{\partial \varphi}{\partial \overline z_j}\,u_j-P_\varphi(\frac{\partial \varphi}{\partial \overline z_j}\,u_j) , \frac{\partial \varphi}{\partial \overline z_k}\,u_k \right )_\varphi,$$
by the last computation in the proof of Theorem \ref{lbasic1}. This term can be written in the form 
\begin{eqnarray*}
 \sum_{j,k=1}^n (R_\varphi v_j, v_k)_\varphi &=& \sum_{j,k=1}^n (R_\varphi v_j, R_\varphi v_k)_\varphi \\
 &=& (R_\varphi v_1+ \dots + R_\varphi v_n , R_\varphi v_1+ \dots + R_\varphi v_n )_\varphi \\
 &=& \| R_\varphi v_1+ \dots + R_\varphi v_n \|^2_\varphi \\
 &=& \|V\|^2_\varphi - \|P_\varphi V\|^2_\varphi,
\end{eqnarray*}
and we are done.

\end{proof}
\begin{rem}
Notice that for $u = \sum_{j=1}^n u_j\, dz_j \in {\text{dom}}(\partial ) \cap {\text{dom}}(\partial^*_\varphi )$ we have 
\begin{equation}\label{formnorm}
\left \| \frac{\partial u_j}{\partial z_k} \right \|^2_\varphi = \left \| \frac{\partial \varphi}{\partial \overline z_k} \, u_j \right \|^2_\varphi - \int_{\mathbb C^n} \frac{\partial^2 \varphi}{\partial z_k \partial \overline z_k}\, |u_j|^2\, e^{-\varphi} \, d\lambda.
\end{equation}
This follows from 
\begin{eqnarray*}
\left \| \frac{\partial u_j}{\partial z_k} \right \|^2_\varphi = (P_\varphi ( \frac{\partial\varphi}{\partial \overline z_k}
\frac{\partial u_j}{\partial z_k}),u_j)_\varphi
&=& ( \frac{\partial\varphi}{\partial \overline z_k}
\frac{\partial u_j}{\partial z_k}),u_j)_\varphi = ( \frac{\partial u_j}{\partial z_k},\frac{\partial\varphi}{\partial  z_k}\, u_j)_\varphi\\
&=&- (u_j, \frac{\partial}{\partial \overline z_k} (\frac{\partial \varphi}{\partial z_k}\, u_j \, e^{-\varphi}))\\
&=& -(u_j,  \frac{\partial^2 \varphi}{\partial z_k \partial \overline z_k}\, u_j)_\varphi + (u_j, \frac{\partial \varphi}{\partial z_k} \, \frac{\partial \varphi}{\partial \overline z_k}\, u_j)_\varphi\\
&=& \left \| \frac{\partial \varphi}{\partial \overline z_k} \, u_j \right \|^2_\varphi - \int_{\mathbb C^n} \frac{\partial^2 \varphi}{\partial z_k \partial \overline z_k}\, |u_j|^2\, e^{-\varphi} \, d\lambda,
\end{eqnarray*}
where we used again that the components $u_j$ are holomorphic.
\end{rem}

\vskip 0.3 cm

Now we  generalize Theorem \ref{lbasic1} for $(p,0)$-forms  $u=\sum_{|J|=p}' u_J\,dz_J$ with coefficients in $A^2(\mathbb{C}^n, e^{-\varphi})$ where $1\le p\le n-1.$ We notice that 
$$\partial u = \sum_{|J|=p}\kern-1pt{}^{\prime}  \, \sum_{j=1}^n \frac{\partial u_J}{\partial  z_j}\, d z_j \wedge d z_J,$$
and 
$$\partial ^*_\varphi u =  \sum_{|K|=p-1}\kern-7pt{}^{\prime}  \, \, \sum_{j=1}^n P_\varphi (\frac{\partial \varphi}{\partial \overline z_j}\, u_{jK})\,dz_K.$$
We obtain
\begin{align*}
\| \partial u \|^2_\varphi+ \| \partial ^* _\varphi u\|^2_\varphi &= \sum_{|J|=|M|=p}\kern-10pt{}^{\prime} \kern8pt   \sum_{j,k=1}^n\, \epsilon_{jJ}^{kM}\, \int_{\mathbb{C}^n}\frac{\partial u_J}{\partial z_j}
\ovli{\frac{\partial u_M}{\partial z_k}}\, e^{-\varphi}\,d\lambda \\
&+  \sum_{|K|=p-1}\kern-7pt{}^{\prime}  \, \, \,  \sum_{j,k=1}^n\, \, \int_{\mathbb{C}^n}\, P_\varphi (\frac{\partial \varphi}{\partial \overline z_j}u_{jK}) \ovli{ P_\varphi (\frac{\partial \varphi}{\partial \overline z_k} u_{kK}})\, e^{-\varphi}\,d\lambda,
\end{align*}
where $ \epsilon_{jJ}^{kM}=0$ if $j\in J$ or $k\in M$ or if ${k} \cup M \neq {j} \cup J,$ and equals the sign of the permutation $\binom{kM}{jJ}$ otherwise.
The right-hand side  of the last formula can be rewritten as 
\begin{equation}\label {komo76}
\sum_{|J|=p}\,^{'} \, \sum_{j=1}^n\, \left \| \frac{\partial u_J}{\partial z_j}\right \|^2_\varphi +
 \sum_{|K|=p-1}\kern-10pt{}^{\prime}  \, \, \,  \sum_{j,k=1}^n \, \, \int_{\mathbb{C}^n}\left ( P_\varphi (\frac{\partial \varphi}{\partial \overline z_j}u_{jK}) \ovli{P_\varphi (\frac{\partial \varphi}{\partial \overline z_k} u_{kK}})\, - \frac{\partial u_{jK}}{\partial  z_k} \ovli{\frac{\partial u_{kK}}{\partial  z_j}} \right )\, e^{-\varphi}\,d\lambda,
\end{equation}
In order to prove this we first consider the (nonzero) terms where $j=k$ (and hence $M=J$). These terms result in the portion of the first sum in \eqref{komo76} where $j\notin J.$ On the other hand, when $j\neq k,$ then $j\in M$ and $k\in J,$ and deletion of $j$ from $M$ and $k$ from $J$ results in the strictly increasing  multi-index $K$ of length $p-1.$ Consequently, these terms can be collected into the second sum in \eqref{komo76} (in the part with the minus sign, we have interchanged the summation indices $j$ and $k$). In this sum, the terms where $j=k$ compensate for the terms in the first sum where $j\in J.$    

Now one can use the same reasoning as in the last proof to get 
\begin{equation}\label{komo77}
\| \partial u \|^2_\varphi + \| \partial ^* u\|^2_\varphi = 
\sum_{|J|=p}\,^{'}  \sum_{j=1}^n\, \left \| \frac{\partial u_J}{\partial z_j}\right \|^2_\varphi 
+ \sum_{|K|=p-1}\kern-10pt{}^{\prime}  \, \, \,  \sum_{j,k=1}^n \, \, \int_{\mathbb{C}^n}
\frac{\partial^2 \varphi}{\partial z_k \partial \overline z_j}\, u_{jK} \overline {u_{kK}} \,  e^{-\varphi}\,d\lambda 
\end{equation}
$$-\sum_{|K|=p-1}\kern-10pt{}^{\prime}  \, \, \,  \sum_{j,k=1}^n \, \,
\left ( \frac{\partial \varphi}{\partial \overline z_j}\,u_{jK} - P_\varphi(\frac{\partial \varphi}{\partial \overline z_j}\,u_{jK}) ,\frac{\partial \varphi}{\partial \overline z_k}\,u_{kK} \right )_\varphi.$$
 
\begin{rem}
For $\varphi (z)= |z_1|^2+ \dots + |z_n|^2$  we obtain 
\begin{equation}\label{komo66}
\| \partial u \|^2 + \| \partial ^* u\|^2 = 
\sum_{|J|=p}\,^{'}  \sum_{j=1}^n\, \left \| \frac{\partial u_J}{\partial z_j}\right \|^2 +
p \sum_{|J|=p}\,^{'}   \, \, \int_{\mathbb{C}^n} 
|u_J|^2\, e^{-|z|^2}\,d\lambda.
\end{equation}
\end{rem}

\vskip 1.5 cm

\section{The $\partial$-Neumann operator on the Fock space}

\vskip 1.5cm

As an immediate consequence of \eqref{eq: komo5} and \eqref{komo66} we get what is called the basic estimates.

\begin{lemma}\label{lbasic20}
Let $1\le p\le n-1$ and let $ u=\sum_{|J|=p}' u_J\,d z_J \in A^2_{(p,0)}(\mathbb{C}^n, e^{-|z|^2})$ and suppose that
 $u\in {\text{dom}}(\partial ) \cap {\text{dom}}(\partial^* ).$ Then
 \begin{equation}\label{basic3}
\|u\|^2 \le \frac{1}{p}\, (\|\partial u\|^2 + \|\partial^*u\|^2) .
 \end{equation}
\end{lemma}

The proof of the last results follows easily from the corresponding results for general Fock spaces, see Theorem \ref{lbasic1} and \eqref{komo66}. 

Now we can use the machinery of the classical $\ovprt$-Neumann operator to show the following results.

 \begin{lemma}\label{ima}
 Both operators $\partial$ and $\partial^*$ have closed range.
 
 If we  endow ${\text{dom}}(\partial ) \cap {\text{dom}}(\partial^* )$ with the graph-norm $ (\|\partial f\|^2 + \|\partial^*f\|^2)^{1/2},$ the dense subspace ${\text{dom}}(\partial ) \cap {\text{dom}}(\partial^* )$ of $A^2_{(p,0)}(\mathbb{C}^n, e^{-|z|^2})$ becomes a Hilbert space. 

 \end{lemma}
 
 \begin{proof}
 We notice that ${\text{ker}}\partial= ({\text{im}} \partial^*)^\perp,$ which implies that
 $$({\text{ker}}\partial )^\perp = \overline{ {\text{im}} \partial^*} \subseteq {\text{ker}}\partial^*.$$
 If $u\in {\text{ker}}\partial \cap {\text{ker}}\partial^*,$ we have by \eqref{basic3} that $u=0.$ Hence
 \begin{equation}\label{eq: ima1}
({\text{ker}}\partial)^\perp = {\text{ker}}\partial^*.
\end{equation}
If $u\in {\text{dom}} (\partial) \cap ({\text{ker}}\partial )^\perp,$ then $u\in {\text{ker}}\partial^*,$ and \eqref{basic3} implies
$$\|u\| \le \frac{1}{p} \, \| \partial u \|.$$
Now we can use general results of unbounded operators on Hilbert spaces (see for instance \cite{Has10} Chapter 4) to show that ${\text{im}}\partial $ and  
${\text{im}}\partial^*$ are closed.
The last assertion follows again by   \eqref{basic3}, see \cite{Has10} Chapter 4.
 \end{proof}
 
 The next result describes the implication of the basic estimates \eqref{basic3} for the $\tilde\Box$-operator.
      
\begin{theorem}\label{sur}   
 The operator  $ \tilde\Box :  {\text{dom}}(\tilde\Box) \longrightarrow A^2_{(p,0)}(\mathbb{C}^n, e^{-|z|^2})$ is bijective and has a bounded inverse 
 $$\tilde N: A^2_{(p,0)}(\mathbb{C}^n, e^{-|z|^2}) \longrightarrow {\text{dom}}(\tilde\Box ). $$
In addition 
  \begin{equation}\label{cont5}
 \|\tilde N u\| \le \frac{1}{p}\, \|u\|,
 \end{equation}
 for each $u\in  A^2_{(p,0)}(\mathbb{C}^n, e^{-|z|^2}).$
 \end{theorem}
 \begin{proof} 
  Since $(\tilde\Box  u,u)=\| \partial u \|^2 + \| \partial ^* u \|^2,$ it follows that for a convergent sequence $(\tilde\Box u_n)_n$ we get
 $$\|\tilde\Box  u_n -\tilde\Box u_m\| \, \|u_n-u_m\| \ge (\tilde\Box  (u_n-u_m),u_n-u_m)\ge  \|u_n-u_m\|^2,$$
 which implies that $(u_n)_n$ is convergent and since $\tilde\Box  $ is a closed operator we obtain that $\tilde\Box  $ has closed range. If $\tilde\Box  u =0,$ we get $\partial u=0$ and $\partial^* u =0$ and by (\ref{basic3}) that $u=0,$ hence $\tilde\Box  $ is injective. Using again general results on unbounded operators on Hilbert spaces we get that the range of $\tilde\Box $ is dense, therefore $\tilde\Box $ is surjective.

  We showed that
 $$ \tilde\Box  :  {\text{dom}}(\tilde\Box ) \longrightarrow A^2_{(p,0)}(\mathbb{C}^n, e^{-|z|^2})$$
 is bijective and therefore has a bounded inverse 
 $$\tilde N: A^2_{(p,0)}(\mathbb{C}^n, e^{-|z|^2}) \longrightarrow {\text{dom}}(\tilde\Box). $$
 For $u\in A^2_{(p,0)}(\mathbb{C}^n, e^{-|z|^2})$ we use \eqref{basic3} for $\tilde N u$ to obtain
 \begin{eqnarray*}
 \|\tilde N u\|^2 &\le & \frac{1}{p}\, (\|\partial \tilde N u\|^2 + \| \partial^* \tilde N u\|^2)\\
 &=& \frac{1}{p}\,((\partial^* \partial \tilde N u,\tilde N u) + (\partial \partial^* \tilde N u,\tilde N u))\\
 &=&\frac{1}{p}\,(u,\tilde N u)\\
 & \le &\frac{1}{p}\,\|u\| \, \|\tilde N u\|,
 \end{eqnarray*}
 which implies
 \eqref{cont5}.
 
\end{proof}

Following the classical $\ovprt$-Neumann calculus we obtain

\begin{theorem}\label{Nprop2}
 Let $\tilde N_p$ denote the inverse of $\tilde \Box $ on $A^2_{(p,0)}(\mathbb{C}^n, e^{-|z|^2}).$ Then
 \begin{equation}\label{comm1}
 \tilde N_{p+1} \partial= \partial \tilde N_p ,
\end{equation}
on ${\text{dom}}(\partial)$ and
\begin{equation}\label{comm2}
 \tilde N_{p-1} \partial^* = \partial^* \tilde N_p,
 \end{equation}
 on ${\text{dom}}(\partial^* ).$ 
 
 In addition we have that $\partial^* \tilde N_p$ is zero on $({\text{ker}} \partial )^\perp .$ 
 \end{theorem}
 
 \begin{proof}
 For $u\in {\text{dom}}(\partial )$ we have $\partial u = \partial \partial^* \partial \tilde N_p u$ and
 $$\tilde N_{p+1} \partial u= \tilde N_{p+1}  \partial \, \partial^* \partial \tilde N_p u = \tilde N_{p+1} (  \partial \, \partial^* +  \partial^*\partial )  \partial \tilde N_p u = \partial \tilde N_p u,$$
 which proves \eqref{comm1}. In a similar way we get  \eqref{comm2}.
 
Now let $k \in ({\text{ker}} \partial )^\perp$ and $u\in {\text{dom}} (\partial ),$ then
$$(\partial^* \tilde N_p k, u)= (\tilde N_p k, \partial u)=(k, \tilde N_p\partial u)=(k, \partial \tilde N_{p-1}u)=0,$$
since $\partial \tilde N_{p-1}u \in {\text{ker}} (\partial ),$ which gives $\partial^* \tilde N_q k=0.$ 
 \end{proof}

Now we can also prove a solution formula for the equation  $\partial u =\alpha ,$ where $\alpha $ is a given $(p,0)$-form in $A^2_{(p,0)}(\mathbb{C}^n, e^{-|z|^2})$ with $\partial \alpha =0.$
\begin{theorem}\label{Nprop3}
 Let $\alpha \in A^2_{(p,0)}(\mathbb{C}^n, e^{-|z|^2})$ with $\partial \alpha =0.$ 
 Then $u_0=\partial^* \tilde N_p \alpha$ is the canonical solution of $\partial u =\alpha, $ this means $\partial u_0 =\alpha $ and $u_0 \in  ({\text{ker}}\, \partial )^\perp = {\text{im}} \, \partial^*,$ and  
 \begin{equation}\label{cont6}
\| \partial^* \tilde N_p\alpha \| \le p^{-1/2} \, \| \alpha \| . 
\end{equation}
\end{theorem}
\begin{proof}
For $\alpha \in A^2_{(p,0)}(\mathbb{C}^n, e^{-|z|^2})$ with $\partial \alpha =0$ 
 we get
\begin{equation}\label{id1}
\alpha = \partial \, \partial^* \tilde N_p \alpha +  \partial^* \, \partial \tilde N_p \alpha  .
\end{equation}
 If we apply $\partial $ to the last equality we obtain:
$$0 =\partial \alpha =\partial \partial^*  \partial  \tilde N_p \alpha ,$$
and since $\partial  \tilde N_p \alpha \in {\text{dom}}(\partial^* ) $ we have
\begin{equation}\label{yd1}
0=(\partial \, \partial ^*  \partial \tilde N_p \alpha, \partial  \tilde N_p \alpha )=( \partial ^*  \partial \tilde N_p \alpha ,  \partial ^*  \partial \tilde N_p \alpha )=\|  \partial ^*  \partial \tilde N_p \alpha \|^2.
\end{equation}
Finally we set $u_0=\partial^* \tilde N_p \alpha$ and derive from \eqref{id1} and \eqref{yd1} that for $\partial \alpha =0$ 
$$\alpha = \partial u_0 ,$$
and we see that $u_0 \bot \,{\text{ker}}\,\partial ,$ since for $h \in {\text{ker}}\,\partial $ we get 
$$(u_0,h)=(\partial^* \tilde N_p \alpha, h )= (\tilde N_p \alpha , \partial h)=0.$$
It follows that
\begin{eqnarray*}
\| \partial^*\tilde N_p \alpha \|^2&=& (\partial \, \partial^*\tilde N_p\alpha, \tilde N_p\alpha)\\
&=&(\partial \, \partial^*\tilde N_p\alpha, \tilde N_p\alpha)+
(\partial^* \partial \tilde N_p\alpha, \tilde N_p\alpha)\\
&=&(\alpha , \tilde N_p\alpha ) \le \|\alpha \| \, \| \tilde N_p\alpha \|
\end{eqnarray*} 
and using \eqref{cont5} we obtain
\begin{equation*}
\| \partial^* \tilde N_p\alpha \| \le p^{-1/2} \, \| \alpha \| .
\end{equation*}
\end{proof}
\vskip 1.5 cm

Now we discuss a different approach to the  operator $\tilde N$ which is related to the quadratic form
$$Q(u,v)= (\partial u,\partial v)+(\partial^*u,\partial^*v).$$
 For this purpose we consider the embedding
$$\iota :  {\text{dom}}(\partial )\cap  {\text{dom}}(\partial^*) \longrightarrow   A^2_{(p,0)}(\mathbb{C}^n, e^{-|z|^2}),$$ 
where $ {\text{dom}}(\partial )\cap  {\text{dom}}(\partial^*)$ is endowed with the graph-norm 
$$u\mapsto (\|\partial u\|^2 + \|\partial^*u \|^2)^{1/2}.$$ 
The graph-norm stems from the inner product
$$Q(u,v)=(u,v)_Q=(\tilde \Box u,v) = (\partial u,\partial v)+(\partial^*u,\partial^*v).$$
The basic estimates \eqref{basic3} imply that $\iota$ is a bounded operator with operator norm 
$$\|\iota \| \le \frac{1}{\sqrt p}.$$ 
By \eqref{basic3} it follows in addition that $ {\text{dom}}(\partial )\cap  {\text{dom}}(\partial^*)$  endowed with the graph-norm $u\mapsto (\|\partial u\|^2 + \|\partial^*u \|^2)^{1/2}$ is a Hilbert space, see Lemma \ref{ima}.

Since $(u,v)=(u,\iota v),$ we have  that $(u,v)=(\iota^*u,v)_Q.$ 

 For $u\in
   A^2_{(p,0)}(\mathbb{C}^n, e^{-|z|^2})$ and $v\in  {\text{dom}}(\partial)\cap  {\text{dom}}(\partial^*)$ we get
\begin{equation}\label{Q1}
   (u,v)=(\tilde \Box \tilde N u,v)=((\partial \partial^* + \partial^* \partial )\tilde Nu,v)=(\partial^* \tilde Nu, \partial^* v)+(\partial \tilde Nu, \partial v).
  \end{equation} 

Equation \eqref{Q1} suggests that  as an operator to ${\text{dom}}(\partial )\cap  {\text{dom}}(\partial^*),$ $\tilde N$ coincides with $\iota^*$ and as an operator to $ A^2_{(p,0)}(\mathbb{C}^n, e^{-|z|^2}), $ 
 $\tilde N$  is equal to $\iota \circ \iota^*,$ see \cite{Has10} or \cite{Str} for the details. 
 
 Hence $\tilde N$ is compact if and only if the embedding
$$\iota :  {\text{dom}}(\partial )\cap  {\text{dom}}(\partial^*) \longrightarrow   A^2_{(p,0)}(\mathbb{C}^n, e^{-|z|^2}),$$ 
is compact. This will be used to prove the following theorem.

\begin{theorem}\label{percol}
The operator $\tilde N :  A^2_{(p,0)}(\mathbb{C}^n, e^{-|z|^2}) \longrightarrow  A^2_{(p,0)}(\mathbb{C}^n, e^{-|z|^2}), \ 1 \le p \le n,$ is compact.
\end{theorem}

\begin{proof} First we consider the case when $p=1.$ We use \eqref{eq: komo5} for the graph norm on 
$ {\text{dom}}(\partial )\cap  {\text{dom}}(\partial^*)$ and indicate that it suffices to consider one component $u_j$ of the $(1,0)$-form $u.$ For this purpose we will denote $u_j$ by $f.$ We have to handle

$$ \sum_{k=1}^n \int_{\mathbb{C}^n}  \left |\frac{\partial f}{\partial z_k}
\right |^2 \,e^{-|z|^2}\,d\lambda +  \int_{\mathbb{C}^n}|f|^2\, e^{-|z|^2}\,d\lambda$$
for the graph-norm. We use the complete orthonormal system  \eqref{vons} $(\varphi_\alpha)_\alpha$ of $A^2(\mathbb{C}^n, e^{-|z|^2}).$
Let $f=\sum_{\alpha} f_\alpha \varphi_\alpha$ be an element of ${\text{dom}}(\partial )\cap  {\text{dom}}(\partial^*).$
We have $\iota (f)=f$ and hence
$$\iota (f) = \sum_{\alpha} (f,\varphi_\alpha)\varphi_\alpha$$
in $A^2(\mathbb{C}^n, e^{-|z|^2}).$
The basis elements $\varphi_\alpha$ are normalized in $A^2(\mathbb{C}^n, e^{-|z|^2}).$
First we have to compute the graph-norm of the basis elements $\varphi_\alpha.$ Notice that 
$$\frac{\partial \varphi_\alpha}{\partial z_k} = \frac{1}{\sqrt{\pi^n}} \, \frac{z_1^{\alpha_1}}{\sqrt{\alpha_1!}}\dots \frac{z_{k-1}^{\alpha_{k-1}}}{\sqrt{\alpha_{k-1}!}} \, \frac{\alpha_k z_k^{\alpha_{k}-1}}{\sqrt{\alpha_k!}} 
\frac{z_{k+1}^{\alpha_{k+1}}}{\sqrt{\alpha_{k+1}!}} \, \dots \frac{z_{n}^{\alpha_{n}}}{\sqrt{\alpha_{n}!}} \,= \sqrt{\alpha_k} \, \varphi_{(\alpha k-1)},$$
where $(\alpha k -1) =(\alpha_1, \dots \alpha_{k-1}, \alpha_k-1, \alpha_{k+1}, \dots , \alpha_n).$
 Hence the graph-norm of the basis elements $\varphi_\alpha$ equals 
$$ \| \varphi_\alpha \|_Q =(\|\varphi_\alpha\|^2 + \sum_{k=1}^n \|\frac{\partial\varphi_\alpha}{\partial z_k}\|^2)^{1/2} =\sqrt{1+|\alpha|},$$
where $|\alpha | = \alpha_1+ \dots + \alpha_n.$

Now let $$\psi_\alpha = \frac{\varphi_\alpha}{\sqrt{1+|\alpha|}}. $$ 
 Then $(\psi_\alpha)_\alpha$ constitutes a complete orthonormal system in  the Hilbert space ${\text{dom}}(\partial )\cap  {\text{dom}}(\partial^*)$ endowed with the graph-norm, notice that
$$(f,\psi_\alpha)_Q= \sum_{k=1}^n(\frac{\partial f}{\partial z_k},\ \frac{\partial \psi_\alpha}{\partial z_k}) + (f,\psi_\alpha)= \sum_{k=1}^n \frac{\alpha_k}{\sqrt{1+|\alpha |}} \, f_\alpha+\frac{1}{\sqrt{1+|\alpha|}} \, f_\alpha = \sqrt{1+|\alpha|} \, f_\alpha,$$
and we have
$$\iota (f)=f = \sum_{\alpha}  (f,\psi_\alpha)_Q \, \psi_\alpha.$$
For the norm of  $A^2(\mathbb{C}^n, e^{-|z|^2})$ we have
\begin{eqnarray*}
\|\iota(f) -  \sum_{|\alpha|\le N}  (f,\psi_\alpha)_Q \, \psi_\alpha \|^2 &=& \|  \sum_{|\alpha|\ge N+1}  (f,\psi_\alpha)_Q \, \psi_\alpha \|^2 \\
&=& \|  \sum_{|\alpha|\ge N+1}  \frac{1}{\sqrt{1+|\alpha|}}\,  (f,\psi_\alpha)_Q \, \varphi_\alpha \|^2 \\
&=& \sum_{|\alpha| \ge N+1} \left | \frac{1}{\sqrt{1+|\alpha|}}\,  (f,\psi_\alpha)_Q \right |^2 \\
&\le & \frac{ \| f\|_Q^2}{N+2} ,
\end{eqnarray*}
where we finally used Bessel's inequality for the Hilbert space  ${\text{dom}}(\partial )\cap  {\text{dom}}(\partial^*)$
endowed with the graph-norm. This proves that 
$$\iota: {\text{dom}}(\partial )\cap  {\text{dom}}(\partial^*) \longrightarrow A^2_{(1,0)}(\mathbb{C}^n, e^{-|z|^2})
$$ 
is a compact operator and the same is true for 
$$\tilde N : A^2_{(1,0)}(\mathbb{C}^n, e^{-|z|^2}) \longrightarrow A^2_{(1,0)}(\mathbb{C}^n, e^{-|z|^2}).$$

For arbitrary $p$ between $1$ and $n$ we can use \eqref{komo66} and the same reasoning as before to get the desired conclusion.
\end{proof}
Compare with the $\ovprt$-Neumann operator $N$ on $L^2(\mathbb C^n, e^{-|z|^2}):$ in this case $N$ fails to be compact, see \cite{Has10}.  This is related to the fact that the kernel of $\ovprt$ is large (it is the Fock space) in case of the $\ovprt$-complex, but the kernel of $\partial$ consists just of the constant functions in case of the $\partial$-complex on the Fock space.

\vskip 0.5 cm

In order to compute the spectrum of the operator $\tilde\Box_p$ we will use the following 
\begin{lemma}\label{davies} Let $A$ be a symmetric operator on a Hilbert space $H$ with domain ${\text{dom}}(A),$ and suppose that $(x_k)_k$ is a complete orthonormal system in $H.$ If each $x_k$ lies in ${\text{dom}}(A),$ and there exist $\lambda_k  \in \mathbb R$ such that 
$$A x_k=\lambda_k x_k$$
for every $k \in \mathbb N, $ then $A$ is essentially self-adjoint and the spectrum of $\overline A$ is the closure in $\mathbb R$ of the set of all $\lambda_k.$
\end{lemma}
See \cite{Dav} or \cite{Has10}.

\begin{theorem}\label{spectrum1}
The spectrum of $\tilde \Box_p,$ where $0\le p \le n,$ consists of all numbers $ m+p$ for $m=0,1,2,\dots,$ where $m+p$ has multiplicity $ { n+m-1 \choose n-1 } { n \choose p }.$
\end{theorem}

\begin{proof}
Recall that the monomials $(\varphi_\alpha)_\alpha,$ where $\alpha = (\alpha_1, \dots, \alpha_n)\in \mathbb N_0^n$ is a multiindex, constitute a complete orthonormal system in $A^2(\mathbb C^n, e^{-|z|^2}) .$ We use \eqref{spec3} and compute
$$\sum_{k=1}^n z_k \frac{\partial \varphi_\alpha}{\partial z_k}+p \varphi_\alpha = (|\alpha |+p) \varphi_\alpha.$$
We use Lemma \ref{davies} and indicate that there are $ { n+|\alpha|-1 \choose n-1 }$ monomials of degree $|\alpha |.$  Hence we get the assertion about the multiplicity from the fact that $A_{(p,0)}^2(\mathbb C^n, e^{-|z|^2})$ is the direct sum of ${ n \choose p }$ copies of $A^2(\mathbb C^n, e^{-|z|^2}) .$
\end{proof}
The last result implies also that the inverse $\tilde N_p$ of $\tilde \Box_p$ is a compact operator with the eigenvalues $1/(m+p).$

Note that the complex Laplacian $\Box_q$ of the $\ovprt$-complex on $L^2(\mathbb C^n, e^{-|z|^2})$ has 
also the eigenvalues $q+m$ for $m=0,1,2,\dots,$ but each of them have infinite multiplicity, see \cite{MaMa},\cite{Hasspec}, \cite{Has10}.
\vskip 1.5 cm

\section{The general $\partial$-complex}
\vskip 1.5cm
Now we return to the classical Fock space but
replace a single derivative with respect to $z_j$ by a differential operator of the form $p_j(\frac{\partial}{\partial z_1}, \dots , \frac{\partial}{\partial z_n}),$ where $p_j$ is a complex polynomial on $\mathbb C^n,$ see \cite{NS1}, \cite{NS2}. We consider the densely defined operators
\begin{equation}\label{gendef1}
Du = \sum_{j=1}^n p_j (u)\, dz_j,
\end{equation}
where $u\in A^2(\mathbb C^n, e^{-|z|^2})$ and $p_j(\frac{\partial}{\partial z_1}, \dots , \frac{\partial}{\partial z_n})$
are polynomial differential operators with constant coefficients.

More general we define
\begin{equation}\label{gendef1'}
Du =  \sum_{|J|=p}\, ' \,   \sum_{k=1}^n p_k (u_J)\, dz_k \wedge dz_J,
\end{equation}
where $u=  \sum_{|J|=p}\, ' \, u_J\, dz_j$ is a $(p,0)$-form with coefficients in $A^2(\mathbb C^n, e^{-|z|^2}).$

It is clear that $D^2=0$ and that we have
\begin{equation}\label{gendef2}
(Du, v) = (u, D^*v),
\end{equation}
where $u\in {\text{dom}}(D)= \{ u \in A^2_{(p,0)}(\mathbb{C}^n, e^{-|z|^2}): Du\in A^2_{(p+1,0)}(\mathbb{C}^n, e^{-|z|^2})\}$ and 
$$D^*v= \sum_{|K|=p-1}\, ' \, \sum_{j=1}^n p_j^*v_{jK}\, dz_K$$
for $v= \sum_{|J|=p}\, ' \, v_J\, dz_J$ and where $p_j^*(z_1, \dots ,z_n)$ is the polynomial $p_j$ with complex conjugate coefficients, taken as multiplication operator.

Now the corresponding $D$-complex has the form
 \begin{equation*}
  A^2_{(p-1,0)}(\mathbb{C}^n, e^{-|z|^2}) 
\underset{\underset{D^* }
\longleftarrow}{\overset{D }
{\longrightarrow}} A^2_{(p,0)}(\mathbb{C}^n, e^{-|z|^2}) \underset{\underset{D^* }
\longleftarrow}{\overset{D }
{\longrightarrow}} A^2_{(p+1,0)}(\mathbb{C}^n, e^{-|z|^2}).
\end{equation*}
In the sequel we consider the generalized box operator
$$\tilde \Box_{D,p}:= D^*D+DD^*$$
as a densely defined self-adjoint operator on $A_{(p,0)}^2(\mathbb C^n, e^{-|z|^2})$ with
 ${\text{dom}} (\tilde\Box_{D,p}) =\{ f\in {\text{dom}}(D) \cap {\text{dom}}(D^* ) : D f \in {\text{dom}}(D^*) \ {\text{and}} \ D^* f \in {\text{dom}}(D )\}.$

We want to find conditions under which $\tilde \Box_{D,1}$ has a bounded inverse. For this purpose we have to consider the graph norm $(\|Du\|^2+ \|D^*u\|^2)^{1/2}$ on ${\text{dom}}(D) \cap {\text{dom}}(D^* ).$

\begin{theorem}\label{basic5}
Let $u= \sum_{j=1}^n u_j dz_j \in {\text{dom}}(D) \cap {\text{dom}}(D^* )$
 and suppose that there exists a constant $C>0$ such that 
 \begin{equation}\label{comm1}
 \|u\|^2 \le C \sum_{j,k=1}^n ( [ p_k, p^*_j ] u_j, u_k).
 \end{equation}
Then 
\begin{equation}\label{basic6}
\|u\|^2 \le C ( \|Du\|^2+ \|D^*u\|^2).
\end{equation}
\end{theorem}
 \begin{proof}
  First we have 
 $$Du= \sum_{j<k} (p_k(u_j)-p_j(u_k))\, dz_j \wedge dz_k \ \ {\text{and}} \ \ 
 D^*u = \sum_{j=1}^n p_j^* u_j,$$
 hence
 $$ \|Du\|^2 + \|D^*u\|^2 = \int_{\mathbb C^n} \sum_{j<k} |p_k(u_j)-p_j(u_k)|^2 \, e^{-|z|^2} \,d\lambda
 + \int_{\mathbb C^n} \sum_{j,k=1}^n p_j^* u_j \, \overline{p_k^*u_k}\, e^{-|z|^2}\,d\lambda$$
 $$=  \sum_{j,k=1}^n \int_{\mathbb C^n} |p_k(u_j)|^2\, e^{-|z|^2}\,d\lambda + 
  \sum_{j,k=1}^n \int_{\mathbb C^n} (p_j^* u_j \, \overline{p_k^*u_k} - p_k(u_j)\overline{p_j(u_k)})\, e^{-|z|^2}\, d\lambda$$
 $$=  \sum_{j,k=1}^n \int_{\mathbb C^n} |p_k(u_j)|^2\, e^{-|z|^2}\,d\lambda +  
  \sum_{j,k=1}^n \int_{\mathbb C^n} [p_k, p_j^*] u_j \overline{u_k}\, e^{-|z|^2}\,d\lambda,$$
where we used \eqref{gendef2}.
Now the assumption \eqref{comm1} implies the desired result.
 \end{proof}

Let $1\le p\le n-1$ and let $ u=\sum_{|J|=p}' u_J\,d z_J \in A^2_{(p,0)}(\mathbb{C}^n, e^{-|z|^2})$ and suppose that
 $u\in {\text{dom}}(D ) \cap {\text{dom}}(D^* ).$
In a similar way as in \eqref{komo77} we get 
\begin{equation}\label{komo88}
\|Du\|^2 + \|D^*u\|^2 = \sum_{|J|=p}\, '  \sum_{k=1}^n \|p_k(u_J)\|^2 + \sum_{|K|=p-1}\, '  \sum_{j,k=1}^n \int_{\mathbb C^n} [p_k, p_j^*] u_{jK} \overline{u_{kK}}\, e^{-|z|^2}\,d\lambda,
\end{equation}
and if we suppose that 
\begin{equation} \label{komo99}
\|u\|^2 \le C  \sum_{|K|=p-1}\, ' \sum_{j,k=1}^n ( [ p_k, p^*_j ] u_{jK}, u_{kK})
\end{equation} 
we get the basic estimate \eqref{basic6}, which also implies that both ${\text{im}}D$ and ${\text{im}}D^*$ are closed, see for instance \cite{Has10}, Chapter 4.
With the basic estimate \eqref{basic6} we are now able to use the machinery of the corresponding Neumann operator - the bounded inverse of $\tilde\Box_{D,p}$ - (see Theorem \ref{sur} and Theorem \ref{Nprop3}) and get the following results

\begin{theorem}\label{genneumann}
Let $D$ be as in \eqref{gendef1'} and suppose that 
$$ \|u\|^2 \le C \sum_{|K|=p-1}\, ' \sum_{j,k=1}^n ( [ p_k, p^*_j ] u_{jK}, u_{kK}),$$
for all $u\in {\text{dom}}(D) \cap {\text{dom}}(D^* ).$ Then $\tilde\Box_{D,p}$ has a bounded inverse
$$\tilde N_{D,p} :  A^2_{(p,0)}(\mathbb{C}^n, e^{-|z|^2}) \longrightarrow  {\text{dom}}(\tilde\Box_{D,p}).$$
If $\alpha \in A^2_{(p,0)}(\mathbb{C}^n, e^{-|z|^2})$ satisfies $D\alpha =0,$ then $u_0= D^* \tilde N_{D,p} \alpha$
is the canonical solution of $Du=\alpha,$ this means $Du_0=\alpha$ and $u_0 \in  ({\text{ker}} D)^\perp = {\text{im}}D^*,$ and
$\| D^*\tilde N_{D,p} \alpha \| \le C \|\alpha\|,$ for some constant $C>0$ independent of $\alpha.$
\end{theorem}
 
\begin{ex}\label{ga1}

a) Let $ p_k = \frac{\partial^2}{\partial z_k^2}.$ Then $p^*_j(z)=z_j^2$ and we have
$$ \sum_{j,k=1}^n ( [ p_k, p^*_j ] u_j, u_k) = \sum_{j,k=1}^n (2\delta_{j,k} u_j,u_k) + \sum_{j,k=1}^n (4\delta_{jk} z_j \frac{\partial u_j}{\partial z_k},u_k)= 2\|u\|^2 + 4 \sum_{j=1}^n \left \| \frac{\partial u_j}{\partial z_j} \right \|^2,$$
for $u= \sum_{j=1}^n u_j dz_j \in {\text{dom}}(D) \cap {\text{dom}}(D^* ).$
Hence \eqref{comm1} is satisfied.

\vskip 0.3 cm
b) Let $n=2$ and take $p_1= \frac{\partial^2}{\partial z_1 \partial z_2}$ and $p_2=  \frac{\partial^2}{\partial z_1^2} +  \frac{\partial^2}{\partial z_2^2}.$ Then $p_1^*(z)= z_1z_2$ and $p_2^*(z)= z_1^2+z_2^2$ and we have
\begin{eqnarray*}
( [ p_1, p^*_1] u_1, u_1)&=& (u_1,u_1)+ (\frac{\partial u_1}{\partial z_1},\frac{\partial u_1}{\partial z_1})+ 
(\frac{\partial u_1}{\partial z_2},\frac{\partial u_1}{\partial z_2}),\\
 ( [ p_1, p^*_2 ] u_2, u_1)&=& 2(\frac{\partial u_2}{\partial z_1},\frac{\partial u_1}{\partial z_2})+
 2(\frac{\partial u_2}{\partial z_2},\frac{\partial u_1}{\partial z_1}),\\
 ( [ p_2, p^*_1 ] u_1, u_2)&=& 2(\frac{\partial u_1}{\partial z_1},\frac{\partial u_2}{\partial z_2})+
 2(\frac{\partial u_1}{\partial z_2},\frac{\partial u_2}{\partial z_1}),\\ 
( [ p_2, p^*_2 ] u_2, u_2)&=& 4(u_2,u_2)+4(\frac{\partial u_2}{\partial z_1},\frac{\partial u_2}{\partial z_1})+4
(\frac{\partial u_2}{\partial z_2},\frac{\partial u_2}{\partial z_2}) .
\end{eqnarray*}
 So we obtain
 
$$ \sum_{j,k=1}^2 ( [ p_k, p^*_j ] u_j, u_k) = \int_{\mathbb C^2}  (|u_1|^2+ 4|u_2|^2
+ \left |\frac{\partial u_1}{\partial z_1}+ 2 \frac{\partial u_2}{\partial z_2}\right |^2 + \left |\frac{\partial u_1}{\partial z_2}
+ 2\frac{\partial u_2}{\partial z_1}\right |^2 )\, e^{-|z|^2}\, d\lambda,
$$

for $u= \sum_{j=1}^2 u_j dz_j \in {\text{dom}}(D) \cap {\text{dom}}(D^* ).$ Again, \eqref{comm1} is satisfied.
\end{ex}
 \vskip 0.3 cm

 We remark that we can interchange the roles of $D$ and $D^*$ and obtain

\begin{theorem}\label{multop}
Suppose that $n>1$ and $1\le p \le n-1.$ Let $D$ be as in \eqref{gendef1'} and suppose that 
$$\|u\|^2 \le C  \sum_{|K|=p-1}\, ' \sum_{j,k=1}^n ( [ p_k, p^*_j ] u_{jK}, u_{kK}),$$
for all $u\in {\text{dom}}(D) \cap {\text{dom}}(D^* ).$ 
If $\beta \in A^2_{(p,0)}(\mathbb{C}^n, e^{-|z|^2})$ satisfies $D^*\beta =0,$ then $v_0= D \tilde N_{D,p} \beta
\in A^2_{(p+1,0)}(\mathbb{C}^n, e^{-|z|^2})$
is the canonical solution of $D^*v=\beta,$ this means $D^*v_0=\beta$ and $v_0\in ({\text{ker}} D^*)^\perp = {\text{im}}D,$ and
$\| D\tilde N_{D,p} \beta \| \le C \|\beta\|,$ for some constant $C>0$ independent of $\beta.$

\end{theorem} 
\begin{proof}
As in Theorem \ref{genneumann} we get that  $\tilde\Box_{D,p}$ has a bounded inverse
$$\tilde N_{D,p} :  A^2_{(p,0)}(\mathbb{C}^n, e^{-|z|^2}) \longrightarrow  {\text{dom}}(\tilde\Box_{D,p}).$$
Now, using the $\ovprt$-Neumann calculus, we obtain that
$$0=D^*\beta = D^* (D^* D+ D D^*)\tilde N_{D,p} \beta = D^* D D^*\tilde N_{D,p} \beta,$$
hence
$$0=(D^* D D^*\tilde N_{D,p} \beta, D^*\tilde N_{D,p} \beta)=(D D^*\tilde N_{D,p} \beta,D D^*\tilde N_{D,p} \beta).$$
This implies that $D D^*\tilde N_{D,p} \beta=0$ and we get
$$D^*v_0= D^* D\tilde N_{D,p} \beta = (DD^*+D^*D)\tilde N_{D,p} \beta = \beta,$$
and $(v_0, f) = (D\tilde N_{D,p} \beta,f) = (\tilde N_{D,p} \beta, D^*f)=0,$ for all $f\in {\text{ker}}D^*.$

\end{proof} 

\begin{ex} We take Example \ref{ga1} b) and consider $f = f_1\,dz_1+ f_2\,dz_2 \in A^2_{(1,0)}(\mathbb{C}^2, e^{-|z|^2})$ such that $D^*f = p_1^* f_1+ p_2^*f_2=0.$ By Theorem \ref{multop} we get 
$$g=g \,dz_1\wedge dz_2 =D\tilde N_{D,1} f \in A^2_{(2,0)}(\mathbb{C}^2, e^{-|z|^2})$$ 
such that $D^*g = -p_2^*g\, dz_1 + p_1^*g\,dz_2= f$ and $\| D\tilde N_{D,1} f \| \le C \|f\|,$ for some constant $C>0.$
\end{ex}

\begin{rem}
Finally we point out that the $\partial$-Neumann operator $\tilde N_{D,p}$ exists and is bounded on $A^2_{(p,0)}(\mathbb{C}^n, e^{-|z|^2})$ if and only if the basic estimate \eqref{basic6} holds, see for instance \cite{Has10}, Remark 9.12., for the details.
\end{rem}

\vskip 2 cm

\bibliographystyle{amsplain}
\bibliography{mybibliography}

\end{document}